\documentclass[11pt,letterpaper,reqno]{amsart}
\usepackage{tikz}
\usetikzlibrary{positioning, shapes.geometric, arrows.meta, calc}
\usepackage{amssymb}
\usepackage{amsmath}
\usepackage{amsthm}
\usepackage{amsfonts}
\usepackage{bbm}
\usepackage{enumitem} 
\usepackage{pgfplots}
\pgfplotsset{compat=1.18} 
\usepackage{booktabs}
\usepackage{graphicx}
\usepackage[T1]{fontenc}
\usepackage{doi}
\usepackage{float} 
\usepackage{bookmark}
\usepackage{hyperref}
\hypersetup{pdfstartview={FitH}}

\newtheorem{theorem}{Theorem}[section]
\newtheorem{lemma}[theorem]{Lemma}
\newtheorem{proposition}[theorem]{Proposition}
\newtheorem{corollary}[theorem]{Corollary}

\newtheorem{problem}[theorem]{Problem}

\theoremstyle{definition}

\numberwithin{equation}{section}

\newcommand{\R}{\mathbb{R}}        

\newcommand{\C}{\mathbb{C}}

\newcommand{\Z}{\mathbb{Z}}

\begin{document}

\title[Generalizing the Clunie--Hayman construction]{Generalizing the Clunie--Hayman construction in an Erd\H{o}s maximum-term problem}

\author[Y.~He]{Yixin He}
\author[Q.~Tang]{Quanyu Tang}

\address{School of Mathematical Sciences, Fudan University, Shanghai 200433, P. R. China}
\email{hyx717math@163.com}
\address{School of Mathematics and Statistics, Xi'an Jiaotong University, Xi'an 710049, P. R. China}
\email{tang\_quanyu@163.com}





\begin{abstract}
Let $f(z)=\sum_{n\ge0}a_n z^n$ be a transcendental entire function and write $M(r,f):=\max_{|z|=r}|f(z)|$ and $\mu(r,f):=\max_{n\ge0}|a_n|\,r^n$. A problem of Erd\H{o}s asks for the value of
$$
B:=\sup_f \liminf_{r\to\infty}\frac{\mu(r,f)}{M(r,f)}.
$$
In 1964, Clunie and Hayman proved that $\frac{4}{7}<B<\frac{2}{\pi}$. In this paper we develop a generalization of their construction via a scaling identity and obtain the explicit lower bound
$$
B>0.58507,
$$
improving the classical constant $\frac{4}{7}$.
\end{abstract}

\maketitle

\section{Introduction}

Let $f(z)=\sum_{n=0}^{\infty}a_n z^n$ be an entire function and put
\[
M(r,f):=\max_{|z|=r}|f(z)|,
\qquad
\mu(r,f):=\max_{n\ge 0}|a_n|\,r^n.
\]Following Erd\H{o}s, we consider the ratio $\mu(r,f)/M(r,f)$, which measures how large the maximal term of the power series can be
compared to the maximum modulus on the circle $|z|=r$.

In~\cite[p.~249]{Er61} (see also~\cite[Problem 2.14(c)]{HaLi18}), Erd\H{o}s proposed the following problem,
which also appears as Problem~\#513 on Bloom's Erd\H{o}s Problems website~\cite{EP513}.

\begin{problem}[\cite{Er61}]\label{prob:EP513}
Let $f(z)=\sum_{n=0}^\infty a_n z^n$ be a transcendental entire function. Determine the constant
\[
B:=\sup_{f}\ \liminf_{r\to\infty}\frac{\mu(r,f)}{M(r,f)}.
\]
\end{problem}

Equivalently, writing
\[
\beta(f):=\liminf_{r\to\infty}\frac{\mu(r,f)}{M(r,f)},
\]
we have $B=\sup\{\beta(f): f \text{ is transcendental entire}\}$, and
Problem~\ref{prob:EP513} asks for the value of $B$.

Erd\H{o}s~\cite{Er61} remarked that $B\in[1/2,1)$, and K\"{o}v\'{a}ri (unpublished) observed that in fact $B>1/2$.
In 1964, Clunie and Hayman~\cite{ClHa64} initiated the first systematic study and proved that
\[
0.57143\approx \frac{4}{7}<B<\frac{2}{\pi}\approx 0.63662.
\]Other results concerning the ratio $\mu(r,f)/M(r,f)$ were obtained by Gray and Shah~\cite{GrSh63}.

The lower-bound construction of Clunie and Hayman~\cite[p.~156]{ClHa64} is based on a remarkable scaling identity. Fix $K>1$ and consider the Laurent series
\[
k(z)=\sum_{n\in\mathbb Z}\frac{(-1)^{n(n-1)/2}}{K^{n(n+1)/2}} z^n,
\]
which defines a holomorphic function on $\mathbb C\setminus\{0\}$ and satisfies
\begin{equation}\label{eq:ClHa-ident-v1}
k(Kz)=z\,k(-z).
\end{equation}
Let
\[
f(z)=\sum_{n\ge 0}\frac{(-1)^{n(n-1)/2}}{K^{n(n+1)/2}} z^n
\]
be the associated entire function obtained by discarding the negative-index terms. Since the negative-index tail of $k$ tends to $0$ as $z\to\infty$, one has $f(z)=k(z)+o(1)$ as $z\to\infty$. Iterating~\eqref{eq:ClHa-ident-v1} yields an explicit formula for $M(K^m,k)$, and hence the estimation of $\beta(f)$ is reduced to controlling the single quantity $\max_{|z|=1}|k(z)|$.

The present paper develops a two-parameter extension of the Clunie--Hayman construction.
We replace the ``sign pattern'' $(-1)^{n(n-1)/2}$ by a unimodular phase $\varepsilon^{n(n-1)/2}$. More precisely, given parameters $K>1$ and $\varepsilon\in\mathbb C$ with $|\varepsilon|=1$, we set\[
f_{K,\varepsilon}(z):=\sum_{n=0}^{\infty}\frac{\varepsilon^{n(n-1)/2}}{K^{n(n+1)/2}} z^n,
\qquad
k_{K,\varepsilon}(z):=\sum_{n\in\mathbb Z}\frac{\varepsilon^{n(n-1)/2}}{K^{n(n+1)/2}} z^n.
\]The Laurent series $k_{K,\varepsilon}$ satisfies the scaling identity
\[
k_{K,\varepsilon}(Kz)=z\,k_{K,\varepsilon}(\varepsilon z),
\]
so that iterating this identity gives an explicit scaling law for $M(K^m,k_{K,\varepsilon})$. A key point is that, for this family, one can compute the maximum term exactly on the geometric sequence $r_m=K^m$,
and reduce $\beta(f_{K,\varepsilon})$ to a $\liminf$ over these radii.
Combining these ingredients with an approximation argument comparing $M(K^m,f_{K,\varepsilon})$ and $M(K^m,k_{K,\varepsilon})$
for large $m$, we obtain an exact formula expressing $\beta(f_{K,\varepsilon})$ in terms of a single quantity on the unit circle:
\[
\beta(f_{K,\varepsilon})=\frac{1}{\max_{|z|=1}|k_{K,\varepsilon}(z)|}.
\]
Thus the problem of producing good lower bounds on $B$ is reduced to finding parameters $(K,\varepsilon)$ for which $\max_{|z|=1}|k_{K,\varepsilon}(z)|$
is small.

Our main result is an explicit improvement over the classical lower bound $4/7$.

\begin{theorem}\label{thm:main}
There exist parameters $K_0>1$ and $|\varepsilon_0|=1$ such that the transcendental entire function
$f_{K_0,\varepsilon_0}$ satisfies
\[
\beta(f_{K_0,\varepsilon_0})>0.58507.
\]
Consequently, $B>0.58507$.
\end{theorem}

The only computational input is a rigorous upper bound for $\max_{|z|=1}|k_{K_0,\varepsilon_0}(z)|$,
obtained using ball arithmetic. Reproducibility details, including code and logs, are provided in Appendix~\ref{app:cert}.

\subsection{Declaration of AI usage}
An AI assistant (ChatGPT, model: GPT-5.2 Pro) was used for exploratory brainstorming and to draft the initial version of the script \texttt{cert\_mesh\_arb.py} in Appendix~\ref{app:cert}. All mathematical arguments and all code were subsequently checked and verified by the human authors.

\subsection{Paper organization}
In Section~\ref{sec:section2} we develop a two-parameter extension of the Clunie--Hayman construction.
We establish the key scaling identity and its iterates, derive an explicit formula for the maximum modulus along the geometric radii $r=K^m$,
determine the maximum term on the same radii, and prove an endpoint reduction for the defining $\liminf$.
Together these steps yield an exact expression that reduces the lower-bound problem to controlling a single maximum on the unit circle.
In Section~\ref{sec:section3} we fix explicit parameters and certify the required bound using interval (ball) arithmetic.
Appendix~\ref{app:cert} documents the certification procedure and provides the code and logs for independent verification.

\section{Scaling identity and an exact formula for \texorpdfstring{$\beta(f_{K,\varepsilon})$}{beta}}\label{sec:section2}

Fix parameters
\[
K>1,\qquad \varepsilon\in\C,\quad |\varepsilon|=1.
\]
For each integer $n\in\Z$ define
\[
T_n:=\frac{n(n+1)}{2}\in\Z,\qquad
A_n:=\frac{\varepsilon^{n(n-1)/2}}{K^{T_n}}.
\]
We consider the entire function
\[
f_{K,\varepsilon}(z):=\sum_{n=0}^{\infty}A_n z^n,
\]
together with the associated bilateral Laurent series
\[
k_{K,\varepsilon}(z):=\sum_{n\in\Z}A_n z^n.
\]
When no confusion can arise we suppress the parameters and write simply $f=f_{K,\varepsilon}$ and $k=k_{K,\varepsilon}$.

\begin{lemma}\label{lem:conv}
The function $f_{K,\varepsilon}$ is transcendental entire, and the Laurent series $k_{K,\varepsilon}$
defines a holomorphic function on $\C\setminus\{0\}$.
\end{lemma}

\begin{proof}
Since $|A_n|=K^{-T_n}=K^{-n(n+1)/2}$ for $n\ge0$, we have
\[
\limsup_{n\to\infty}|A_n|^{1/n}
=\lim_{n\to\infty}K^{-(n+1)/2}=0,
\]
so $f$ has infinite radius of convergence by Cauchy--Hadamard and is entire. As $A_n\neq0$ for all $n\ge0$,
it is not a polynomial and hence is transcendental.

For $k$, fix an annulus $\rho\le |z|\le R$ with $0<\rho<R<\infty$.
For $n\ge0$ we have $|A_n z^n|\le K^{-n(n+1)/2}R^n$, and for $n=-m\le -1$,
using $T_{-m}=m(m-1)/2$,
\[
|A_{-m}z^{-m}|\le K^{-m(m-1)/2}\rho^{-m}.
\]
Both majorant series converge by the ratio test, hence the Laurent series converges absolutely and uniformly
on $\{\rho\le |z|\le R\}$ by the Weierstrass $M$-test. Since $\rho,R$ were arbitrary, $k$ is holomorphic on $\C\setminus\{0\}$.
\end{proof}

The coefficients were chosen so that the Laurent series $k$ satisfies a simple scaling identity.
This is the key mechanism in the argument of Clunie and Hayman~\cite{ClHa64}, reducing the estimation of $\beta(f)$ to a bound for $\max_{|z|=1}|k(z)|$.

\begin{lemma}\label{lem:functional}
The function $k_{K,\varepsilon}$ satisfies
\[
k_{K,\varepsilon}(Kz)=z\,k_{K,\varepsilon}(\varepsilon z),\qquad 0<|z|<\infty.
\]
\end{lemma}

\begin{proof}
On any annulus $0<\rho\le |z|\le R<\infty$ the Laurent series converges absolutely and uniformly, so we may compute termwise:
\[
k(Kz)=\sum_{n\in\Z}A_n(Kz)^n=\sum_{n\in\Z}A_nK^n z^n.
\]
Using $T_n-n=n(n-1)/2$, we have $A_nK^n=\varepsilon^{n(n-1)/2}/K^{n(n-1)/2}$, hence
\[
k(Kz)=\sum_{n\in\Z}\frac{\varepsilon^{n(n-1)/2}}{K^{n(n-1)/2}}\,z^n.
\]
On the other hand,
\[
z\,k(\varepsilon z)=z\sum_{n\in\Z}A_n(\varepsilon z)^n
=\sum_{n\in\Z}A_n\varepsilon^n z^{n+1}
=\sum_{n\in\Z}A_{n-1}\varepsilon^{n-1}z^{n}.
\]
Since $T_{n-1}=n(n-1)/2$, we get $A_{n-1}\varepsilon^{n-1}=\varepsilon^{n(n-1)/2}/K^{n(n-1)/2}$,
which matches the coefficient in $k(Kz)$.
\end{proof}

Define
\[
A(K,\varepsilon):=\max_{|z|=1}\bigl|k_{K,\varepsilon}(z)\bigr|.
\]Iterating Lemma~\ref{lem:functional} yields an explicit formula for the maximum modulus of $k_{K,\varepsilon}$
on the circles $|z|=K^m$.

\begin{lemma}\label{lem:iterates}
For each integer $m\ge 1$, one has the exact identity
\[
M(K^m,k_{K,\varepsilon})=K^{m(m-1)/2}\,A(K,\varepsilon).
\]
\end{lemma}
\begin{proof}
We prove by induction on $m$ the stronger identity
\begin{equation}\label{eq:iterate-exact}
k_{K,\varepsilon}(K^m z)
=K^{m(m-1)/2}\,\varepsilon^{m(m-1)/2}\, z^m\, k_{K,\varepsilon}(\varepsilon^m z),
\qquad z\neq 0.
\end{equation}
For $m=1$, \eqref{eq:iterate-exact} is exactly Lemma~\ref{lem:functional}.
Assume \eqref{eq:iterate-exact} holds for some $m\ge 1$. Using Lemma~\ref{lem:functional} at $K^m z$ and the commutativity of scalar multiplication,
\[
k_{K,\varepsilon}(K^{m+1}z)
= k_{K,\varepsilon}(K(K^m z))
=K^m z\, k_{K,\varepsilon}(\varepsilon K^m z)
=K^m z\, k_{K,\varepsilon}(K^m(\varepsilon z)).
\]
Applying the induction hypothesis to $\varepsilon z$ gives
\[
k_{K,\varepsilon}(K^m(\varepsilon z))
=K^{m(m-1)/2}\,\varepsilon^{m(m-1)/2}\,(\varepsilon z)^m\,
k_{K,\varepsilon}(\varepsilon^{m+1} z).
\]
Combining the last two displays yields \eqref{eq:iterate-exact} for $m+1$.

Taking moduli in \eqref{eq:iterate-exact} and using $|\varepsilon|=1$ yields
\[
|k_{K,\varepsilon}(K^m z)|
=K^{m(m-1)/2}\,|z|^m\,|k_{K,\varepsilon}(\varepsilon^m z)|.
\]
Now setting $|z|=1$ and maximizing over $|z|=1$ gives
$M(K^m,k_{K,\varepsilon})=K^{m(m-1)/2}A(K,\varepsilon)$.
\end{proof}

To convert the explicit formula for $M(K^m,k_{K,\varepsilon})$ in Lemma~\ref{lem:iterates} into information about $\beta(f_{K,\varepsilon})$,
we need two further inputs: (i) an exact description of the maximum term $\mu(r,f_{K,\varepsilon})$ at the same radii $r=K^m$, and (ii) an endpoint reduction showing that the $\liminf$ defining $\beta(f_{K,\varepsilon})$ may be taken along $r_m=K^m$.

\begin{lemma}\label{lem:mu}
Let $r_m:=K^m$ for $m\ge 0$. Then
\[
\mu(r_m,f_{K,\varepsilon})=K^{m(m-1)/2}\qquad (m\ge 0).
\]
Moreover, for $K^m<r<K^{m+1}$ the maximum term is uniquely attained at index $m$,
while at the endpoint $r=K^m$ (for $m\ge 1$) the two indices $m-1$ and $m$ tie for the maximum.
\end{lemma}

\begin{proof}
Since $|A_n|r^n=K^{-n(n+1)/2}\,r^n$, the ratio of consecutive terms is
\[
\frac{|A_{n+1}|r^{n+1}}{|A_n|r^n}=\frac{r}{K^{n+1}}.
\]
If $K^m<r<K^{m+1}$, this ratio is $>1$ for $n\le m-1$ and $<1$ for $n\ge m$, so the terms increase up to $n=m$
and decrease thereafter, giving a unique maximiser $n=m$. If $r=K^m$, the ratio equals $1$ precisely at $n=m-1$ and the maximum is
attained exactly at $n=m-1$ and $n=m$. In particular,
\[
\mu(r_m,f_{K,\varepsilon})=|A_m|K^{m^2}
=K^{-m(m+1)/2}K^{m^2}
=K^{m(m-1)/2}.
\qedhere
\]
\end{proof}

\begin{proposition}\label{prop:endpoint-K}
For $f_{K,\varepsilon}$ one may take $r_m=K^m$ in the definition of $\beta(f_{K,\varepsilon})$. In particular,
\[
\beta(f_{K,\varepsilon})=\liminf_{m\to\infty}\frac{\mu(K^m,f_{K,\varepsilon})}{M(K^m,f_{K,\varepsilon})}.
\]
\end{proposition}

\begin{proof}
Fix $m\ge0$ and consider $r\in(K^m,K^{m+1})$. By Lemma~\ref{lem:mu}, the maximum term is uniquely attained at index $m$, hence
$\mu(r,f)=|A_m|r^m$, so the function $t=\log r\mapsto \log\mu(e^t,f)=\log|A_m|+mt$ is affine on $(\log K^m,\log K^{m+1})$.
By Hadamard's three-circles theorem, $t\mapsto\log M(e^t,f)$ is convex. Therefore \(t\mapsto \log\left(\mu(e^t,f)/M(e^t,f)\right)\) is concave on $(\log K^m,\log K^{m+1})$, so $\mu(r,f)/M(r,f)$ attains its minimum on $[K^m,K^{m+1}]$ at an endpoint.
Consequently, for each $N$,
\[
\inf_{r\ge r_N}\frac{\mu(r,f)}{M(r,f)}
=\inf_{m\ge N}\ \inf_{r\in[r_m,r_{m+1}]}\frac{\mu(r,f)}{M(r,f)}
=\inf_{m\ge N}\ \min\left\{\frac{\mu(K^m,f)}{M(K^m,f)},\,\frac{\mu(K^{m+1},f)}{M(K^{m+1},f)}\right\}.
\]Taking $N\to\infty$ and using $\beta(f)=\lim_{N\to\infty}\inf_{r\ge r_N}\mu(r,f)/M(r,f)$ yields
\[
\beta(f)=\lim_{N\to\infty}\inf_{m\ge N}\ \min\left\{\frac{\mu(K^m,f)}{M(K^m,f)},\,\frac{\mu(K^{m+1},f)}{M(K^{m+1},f)}\right\}=\liminf_{m\to\infty}\frac{\mu(K^m,f)}{M(K^m,f)}.
\qedhere\]
\end{proof}

Finally, $f$ is obtained from $k$ by discarding the negative-index tail. The next two lemmas record that this tail is negligible on large circles.

\begin{lemma}\label{lem:Mmineq}
Let $F$ and $G$ be continuous on the circle $\{|z|=r\}$. Then
\[
\bigl|M(r,F)-M(r,G)\bigr|\le M(r,F-G).
\]
\end{lemma}

\begin{proof}
For $|z|=r$, the triangle inequality gives $|F(z)|\le |G(z)|+|F(z)-G(z)|$, hence
$M(r,F)\le M(r,G)+M(r,F-G)$. Interchanging $F$ and $G$ yields the reverse inequality, and the claim follows.
\end{proof}

\begin{lemma}\label{lem:neg}
As $|z|\to\infty$,
\[
k_{K,\varepsilon}(z)-f_{K,\varepsilon}(z)=O\left(\frac{1}{|z|}\right).
\]
\end{lemma}

\begin{proof}
We have
\[
k(z)-f(z)=\sum_{n=-\infty}^{-1}A_n z^n
=A_{-1}z^{-1}+\sum_{m=2}^{\infty}A_{-m}z^{-m}.
\]
Since $T_{-1}=0$, we have $A_{-1}=\varepsilon$, so the first term is $\varepsilon z^{-1}$.
Moreover, $T_{-m}=m(m-1)/2$ for $m\ge2$, hence the series $\sum_{m\ge2}|A_{-m}|=\sum_{m\ge2}K^{-m(m-1)/2}$ converges.
For $|z|\ge1$ this gives
\[
\left|\sum_{m=2}^{\infty}A_{-m}z^{-m}\right|
\le |z|^{-2}\sum_{m=2}^{\infty}|A_{-m}|
=O(|z|^{-2}),
\]
so $k(z)-f(z)=\varepsilon z^{-1}+O(|z|^{-2})=O(1/|z|)$ as $|z|\to\infty$.
\end{proof}

\begin{theorem}\label{thm:exactbeta}
We have
\[
\beta(f_{K,\varepsilon})=\frac{1}{A(K,\varepsilon)}.
\]
Consequently, for every choice of $(K,\varepsilon)$ we have the lower bound $B\ge 1/A(K,\varepsilon)$.
\end{theorem}

\begin{proof}
Let $r_m=K^m$. Since $k_{K,\varepsilon}$ is not identically zero, $A(K,\varepsilon)>0$. Lemma~\ref{lem:iterates} gives \(M(r_m,k)=K^{m(m-1)/2}\,A(K,\varepsilon)\), while Lemma~\ref{lem:mu} gives $\mu(r_m,f)=K^{m(m-1)/2}$.
By Lemma~\ref{lem:Mmineq} and Lemma~\ref{lem:neg},
\[
\bigl|M(r_m,f)-M(r_m,k)\bigr|\le M(r_m,f-k)=O(K^{-m})=o\bigl(\mu(r_m,f)\bigr).
\]
Hence $M(r_m,f)=K^{m(m-1)/2}A(K,\varepsilon)+o(K^{m(m-1)/2})$, and therefore
\[
\frac{\mu(r_m,f)}{M(r_m,f)}=\frac{K^{m(m-1)/2}}{K^{m(m-1)/2}A(K,\varepsilon)+o(K^{m(m-1)/2})}=\frac{1}{A(K,\varepsilon)}+o(1).
\]
Using Proposition~\ref{prop:endpoint-K} yields $\beta(f_{K,\varepsilon})=1/A(K,\varepsilon)$.
Finally, Lemma~\ref{lem:conv} shows $f_{K,\varepsilon}$ is transcendental entire, so $B\ge 1/A(K,\varepsilon)$.
\end{proof}

\subsection{Theta-function viewpoint}
The material in this subsection is not used elsewhere in the paper.
It is included only to identify $k_{K,\varepsilon}$ with a classical special function, a perspective that can be
useful for organizing and accelerating the subsequent parameter search.
We follow the standard notation for Ramanujan's general theta function (see, e.g., \cite[\S1.6]{GaRa04}).
Let $q:=K^{-1}\in(0,1)$ and define Ramanujan's theta function%
\footnote{For a quick overview, see \href{https://en.wikipedia.org/wiki/Ramanujan_theta_function}{\emph{Ramanujan theta function}} on Wikipedia.}
\[
\Phi(a,b):=\sum_{n=-\infty}^{\infty} a^{n(n+1)/2}\, b^{n(n-1)/2},
\qquad |ab|<1.
\]

\begin{proposition}\label{prop:theta}
We have the exact identity
\[
k_{K,\varepsilon}(z)=\Phi(qz,\varepsilon z^{-1}),\qquad 0<|z|<\infty.
\]
\end{proposition}
\begin{proof}
Compute directly:
\[
\Phi(qz,\varepsilon z^{-1})
=\sum_{n\in\Z}(qz)^{n(n+1)/2}(\varepsilon z^{-1})^{n(n-1)/2}
=\sum_{n\in\Z}\frac{\varepsilon^{n(n-1)/2}}{K^{T_n}}z^n
=k_{K,\varepsilon}(z),
\]
since $T_n-n(n-1)/2=n$ and $q^{T_n}=K^{-T_n}$.
\end{proof}

Define the best constant obtainable within this scaling-identity family:
\[
\beta_{\mathrm{SI}}:=\sup_{K>1,\ |\varepsilon|=1}\beta(f_{K,\varepsilon}).
\]

\begin{corollary}\label{cor:extremal}
We have
\[
\beta_{\mathrm{SI}}
=\frac{1}{\displaystyle \inf_{0<q<1,\ |\varepsilon|=1}\ \max_{|z|=1}\bigl|\Phi(qz,\varepsilon z^{-1})\bigr| }.
\]
In particular, \(B\ge \beta_{\mathrm{SI}}\).
\end{corollary}

\begin{proof}
This is an immediate reformulation of Theorem~\ref{thm:exactbeta} using Proposition~\ref{prop:theta} and the identity
$\sup(1/X)=1/\inf(X)$ for positive $X$.
\end{proof}

\section{An explicit numerical lower bound}\label{sec:section3}

By Theorem~\ref{thm:exactbeta}, obtaining an explicit lower bound for $B$ reduces to bounding
\[
A(K,\varepsilon)=\max_{|z|=1}\bigl|k_{K,\varepsilon}(z)\bigr|
\]
for a suitable choice of parameters $(K,\varepsilon)$. In this section we fix explicit parameters $(K_0,\varepsilon_0)$ suggested by a numerical optimization
and compute a fully auditable upper bound for $A(K_0,\varepsilon_0)$ using ball arithmetic.

\begin{lemma}\label{lem:cos}
For every $\theta\in\mathbb R$ one has
\[
e^{i\theta}k_{K,\varepsilon}(\varepsilon e^{2i\theta})
=
2\sum_{n=0}^{\infty}\frac{\varepsilon^{T_n}}{K^{T_n}}\cos\bigl((2n+1)\theta\bigr).
\]
Consequently,
\[
\max_{|z|=1}|k_{K,\varepsilon}(z)|
=
\max_{\theta\in[0,2\pi)}\left|
2\sum_{n=0}^{\infty}\frac{\varepsilon^{T_n}}{K^{T_n}}\cos\bigl((2n+1)\theta\bigr)
\right|.
\]
\end{lemma}
\begin{proof}
Write $z=\varepsilon e^{2i\theta}$. By Lemma~\ref{lem:conv}, the Laurent series defining $k_{K,\varepsilon}$
converges absolutely on $|z|=1$, and hence its terms may be rearranged freely (see~\cite[Theorem~3.55]{Ru76}).
In particular, pairing the terms with indices $n$ and $-n-1$ gives
\[
k(z)=\sum_{n\ge 0}\Bigl(A_n z^n + A_{-n-1}z^{-n-1}\Bigr).
\]
Using $T_{-n-1}=T_n$, one checks
\[
A_n z^n=\frac{\varepsilon^{T_n}}{K^{T_n}}e^{2in\theta},
\qquad
A_{-n-1}z^{-n-1}=\frac{\varepsilon^{T_n}}{K^{T_n}}e^{-2i(n+1)\theta}.
\]
Hence
\[
e^{i\theta}\Bigl(A_n z^n + A_{-n-1}z^{-n-1}\Bigr)
=\frac{\varepsilon^{T_n}}{K^{T_n}}\Bigl(e^{i(2n+1)\theta}+e^{-i(2n+1)\theta}\Bigr)
=2\frac{\varepsilon^{T_n}}{K^{T_n}}\cos((2n+1)\theta),
\]
and summing over $n\ge0$ gives the first identity.

Finally, the map $\theta\mapsto \varepsilon e^{2i\theta}$ is a parametrization of the unit circle, and $|e^{i\theta}|=1$.
Taking absolute values and maxima over $\theta\in[0,2\pi)$ gives the second identity.
\end{proof}

We now fix the explicit parameters used in Theorem~\ref{thm:main}.

Fix the exact rationals
\[
K_0:=\frac{7137}{2000}=3.56850,\qquad \alpha_0:=\frac{198074929}{50000000}=3.96149858,
\]and define $\varepsilon_0:=e^{i\alpha_0}$. Let $f_0:=f_{K_0,\varepsilon_0}$ and $k_0:=k_{K_0,\varepsilon_0}$, and set
\[
A_0:=\max_{|z|=1}|k_0(z)|.
\]
By Theorem~\ref{thm:exactbeta}, $\beta(f_0)=1/A_0$, so it suffices to prove $A_0<1.70919$.

Define the truncated trigonometric polynomial
\[
P(\theta):=
2\sum_{n=0}^{5}\frac{\varepsilon_0^{T_n}}{K_0^{T_n}}\cos\bigl((2n+1)\theta\bigr),
\]
and let
\[
R(\theta):=
2\sum_{n=6}^{\infty}\frac{\varepsilon_0^{T_n}}{K_0^{T_n}}\cos\bigl((2n+1)\theta\bigr)
\]
be the tail, so that Lemma~\ref{lem:cos} gives
\[
A_0=\max_{\theta\in[0,2\pi)}|P(\theta)+R(\theta)|.
\]

\begin{lemma}\label{lem:tail}
For all $\theta\in\R$,
\[
|R(\theta)|
\le \frac{2K_0^{-21}}{1-K_0^{-7}}.
\]
\end{lemma}

\begin{proof}
Using $|\varepsilon_0^{T_n}|=1$ and $|\cos|\le 1$,
\[
|R(\theta)|\le 2\sum_{n=6}^{\infty}K_0^{-T_n}.
\]
Since $T_6=21$ and $T_{n+1}-T_n=n+1\ge 7$ for $n\ge 6$, we have $T_n\ge 21+7(n-6)$ for $n\ge6$, hence
\[
\sum_{n=6}^{\infty}K_0^{-T_n}\le K_0^{-21}\sum_{m=0}^{\infty}K_0^{-7m}
=\frac{K_0^{-21}}{1-K_0^{-7}}.
\qedhere
\]
\end{proof}

\begin{lemma}\label{lem:lipschitz}
The function $P$ is $4$-Lipschitz on $\mathbb R$, i.e.
\[
|P(\theta)-P(\varphi)|\le 4|\theta-\varphi|
\qquad(\theta,\varphi\in\mathbb R).
\]
\end{lemma}
\begin{proof}
Differentiate termwise and use $|\sin|\le 1$ and $|\varepsilon_0^{T_n}|=1$:
\[
|P'(\theta)|
\le 2\sum_{n=0}^{5}(2n+1)\,K_0^{-T_n}.
\]
To see this is $<4$, note that $T_0,T_1,\dots,T_5=0,1,3,6,10,15$ and $K_0>3.5$, so
\[
2\sum_{n=0}^{5}(2n+1)K_0^{-T_n}
<
2\left(1+\frac{3}{3.5}+\frac{5}{3.5^3}+\frac{7}{3.5^6}+\frac{9}{3.5^{10}}+\frac{11}{3.5^{15}}\right)
<4.
\]
The Lipschitz bound follows.
\end{proof}

\begin{lemma}\label{lem:mesh}
Let $M:=2{,}000{,}000$ and $\theta_j:=2\pi j/M$ for $j=0,1,\dots,M-1$.
Then the finite maximum satisfies the certified inequality
\[
\max_{0\le j<M}|P(\theta_j)|\le 1.709176398.
\]
\end{lemma}
\begin{proof}
This is a finite computer-assisted verification carried out with Arb ball arithmetic (via \texttt{python-flint}).
For each mesh point $\theta_j$ we evaluate $P(\theta_j)$ in rigorous complex ball arithmetic,
obtaining an enclosure of the exact complex value in the form of an \texttt{acb} ball.
From this enclosure we extract a certified upper bound for the modulus,
denoted in the code by
\[
U_j:=\texttt{abs\_upper}\bigl(P(\theta_j)\bigr),
\]
which is guaranteed to satisfy $U_j\ge |P(\theta_j)|$.
Taking the maximum over $j=0,1,\dots,M-1$ gives
\[
\max_{0\le j<M}|P(\theta_j)|\le \max_{0\le j<M} U_j \le 1.709176398,
\]
as claimed. Full reproducibility details (code, parameters, and output log) are provided in Appendix~\ref{app:cert}.
\end{proof}

\begin{proposition}\label{prop:A0}
One has $A_0<1.70919$.
\end{proposition}

\begin{proof}
Let $M$ and $\theta_j$ be as in Lemma~\ref{lem:mesh}. We bound separately the truncated part and the tail:
\[
A_0=\max_{\theta\in[0,2\pi)}|P(\theta)+R(\theta)|
\le \max_{\theta\in[0,2\pi)}|P(\theta)|+\sup_{\theta\in[0,2\pi)}|R(\theta)|.
\]

\smallskip
\noindent\underline{\emph{Truncated part.}}
Fix $\theta\in[0,2\pi)$. Since the mesh points $\theta_j=2\pi j/M$ have spacing $2\pi/M$,
there exists $j\in\{0,1,\dots,M-1\}$ such that the circular distance satisfies
\[
d(\theta,\theta_j):=\min\bigl(|\theta-\theta_j|,\ 2\pi-|\theta-\theta_j|\bigr)\le \frac{\pi}{M}.
\]
Because $P$ is $2\pi$-periodic, we may replace $\theta$ by $\theta\pm 2\pi$ if necessary so that
$|\theta-\theta_j|=d(\theta,\theta_j)$. Lemma~\ref{lem:lipschitz} then gives
\[
|P(\theta)|\le |P(\theta_j)|+4\,d(\theta,\theta_j)\le |P(\theta_j)|+4\frac{\pi}{M}.
\]
Taking the maximum over $\theta\in[0,2\pi)$ and using Lemma~\ref{lem:mesh}, we obtain
\[
\max_{\theta\in[0,2\pi)}|P(\theta)|
\le \max_{0\le j<M}|P(\theta_j)|+4\frac{\pi}{M}
\le 1.709176398+4\frac{\pi}{2{,}000{,}000}
<1.709183.
\]

\smallskip
\noindent\underline{\emph{Tail.}}
By Lemma~\ref{lem:tail},
\[
\sup_{\theta\in[0,2\pi)}|R(\theta)|
\le \frac{2K_0^{-21}}{1-K_0^{-7}}
<5.1\times 10^{-12}.
\]
Therefore,
\[
A_0<1.709183+5.1\times 10^{-12}<1.70919,
\]
as claimed.
\end{proof}

We are now ready to present the following.

\begin{proof}[Proof of Theorem~\ref{thm:main}]
By Theorem~\ref{thm:exactbeta} we have $\beta(f_0)=1/A_0$. Proposition~\ref{prop:A0} gives
$A_0<170919/100000$, and hence
\[
\beta(f_0)=\frac{1}{A_0}>\frac{100000}{170919}>\frac{58507}{100000}.
\]Therefore $\beta(f_0)>0.58507$. Since $f_0$ is transcendental entire, we have
$B=\sup_f \beta(f)\ge \beta(f_0)>0.58507$.
\end{proof}

\appendix

\section{Auditable certified computation using ball arithmetic}\label{app:cert}

This appendix documents the finite, auditable computation behind Lemma~\ref{lem:mesh}.
The statement to be certified is purely discrete: we evaluate a fixed trigonometric polynomial $P$
on the uniform mesh $\{\theta_j\}_{j=0}^{M-1}$ with $M=2{,}000{,}000$ and obtain a rigorous upper bound for \(\max_{0\le j<M}|P(\theta_j)|\).

\subsection{Setup: the exact objects being evaluated}

We fix the exact parameters
\[
K_0=\frac{7137}{2000},\qquad \alpha_0=\frac{198074929}{50000000},\qquad \varepsilon_0=e^{i\alpha_0},
\]
and for $n=0,1,\dots,5$ define
\[
T_n=\frac{n(n+1)}{2},\qquad
c_n=\frac{\varepsilon_0^{T_n}}{K_0^{T_n}}=\frac{e^{i\alpha_0 T_n}}{K_0^{T_n}}.
\]
We then consider the trigonometric polynomial
\[
P(\theta)=2\sum_{n=0}^{5} c_n \cos\bigl((2n+1)\theta\bigr).
\]
Finally, set
\[
M=2{,}000{,}000,\qquad \theta_j=\frac{2\pi j}{M}\qquad (j=0,1,\dots,M-1).
\]
Lemma~\ref{lem:mesh} claims the certified inequality
\[
\max_{0\le j<M} |P(\theta_j)| \le 1.709176398.
\]

\subsection{Ball arithmetic: what is guaranteed}

Arb-style ball arithmetic represents a real number by a midpoint-radius interval
\[
[m\pm r]=\{x\in\R:\ |x-m|\le r\},
\]
and similarly represents a complex number by a complex ball.
The key guarantee is the \emph{outward-rounding principle}:

\medskip
\noindent\textbf{Outward rounding principle.}
\emph{Each elementary operation (addition, multiplication, $\cos$, $\exp$, etc.) returns a ball that
provably contains the exact result corresponding to the exact values contained in the input balls;}
see~\cite{johansson-arb}.

\medskip
In particular, if a complex ball $Z$ is known to contain the exact value $P(\theta_j)$, then any certified
upper bound for $\sup\{|z|:z\in Z\}$ is automatically a rigorous upper bound for $|P(\theta_j)|$.

\subsection{Certification of Lemma~\ref{lem:mesh}: the finite procedure}

The certification is a deterministic finite computation. Conceptually, it consists of the following steps.

\begin{enumerate}
\item \textbf{Exact input (no floats).}
Represent $K_0$ and $\alpha_0$ as exact multiprecision rationals (type \texttt{fmpq}), and only then convert them
to Arb real balls via \texttt{arb(fmpq(...))}. This avoids any ambiguity from floating-point parsing; see
\cite{pyflint-fmpq,pyflint-general}.

\item \textbf{Rigorous $\pi$ and rigorous mesh points.}
Obtain $\pi$ as a real ball guaranteed to contain the true value of $\pi$.
Compute each $\theta_j=(2\pi/M)\,j$ as the product of the $\pi$-ball with the exact rational $2j/M$.
The resulting real ball is guaranteed to contain the exact real number $\theta_j$.

\item \textbf{Rigorous coefficients.}
For each $n=0,1,\dots,5$, compute
\[
c_n=\frac{e^{i\alpha_0 T_n}}{K_0^{T_n}}
\]
using rigorous complex exponentiation and division, producing a complex ball containing the exact $c_n$.

\item \textbf{Rigorous evaluation of $P(\theta_j)$.}
For each $j=0,1,\dots,M-1$, evaluate
\[
P(\theta_j)=2\sum_{n=0}^{5} c_n \cos\bigl((2n+1)\theta_j\bigr)
\]
using rigorous cosine and ball operations. The output is a complex ball enclosing the exact value $P(\theta_j)$.

\item \textbf{Rigorous modulus bounds and the discrete maximum.}
From the enclosing complex ball $Z_j\ni P(\theta_j)$, extract a certified upper bound
\[
u_j \ge |P(\theta_j)|.
\]
In our implementation this is done by \texttt{abs\_upper}: the value \texttt{abs\_upper}$(Z_j)$ is guaranteed to be
$\ge \sup\{|z|: z\in Z_j\}$, and hence $\ge |P(\theta_j)|$; see~\cite{pyflint-acb}.
Finally set
\[
U:=\max_{0\le j<M}u_j.
\]
Then $U$ is a rigorous upper bound for $\max_{0\le j<M}|P(\theta_j)|$.

\item \textbf{Comparison with the target bound.}
Let
\[
C_{\mathrm{mesh}}:=1.709176398=\frac{1709176398}{10^9}.
\]
We certify $U\le C_{\mathrm{mesh}}$ using the ball-comparison routine.
In \texttt{python-flint}, comparisons are conservative: the predicate $U\le C_{\mathrm{mesh}}$ returns \texttt{True}
only when the inequality is certainly valid for the underlying enclosures; see~\cite{pyflint-general}.
Thus, if the program prints \texttt{True}, Lemma~\ref{lem:mesh} is certified.
\end{enumerate}

\subsection{Reproducibility and recorded output}

The above procedure is implemented in the ancillary Python script \texttt{cert\_mesh\_arb.py}. The source code is available in the second author's public repository~\cite{Git}.
The script uses Arb ball arithmetic via the \texttt{arb}/\texttt{acb} types provided by \texttt{python-flint},
and takes as input the mesh size $M$, the working precision (in decimal digits), and the number of worker processes.

\medskip
\noindent\emph{Software versions.}
All certifications reported here were run on a 64-bit Windows machine using \texttt{Python~3.12.3}
and \texttt{python-flint~0.8.0} (FLINT/Arb ball arithmetic).

\medskip
Running the script with
\[
(M,\mathrm{dps},\mathrm{workers})=(2000000,90,1)
\]
produces the following output:
\begin{verbatim}
=== Arb-certified mesh bound ===
M = 2000000, dps = 90, workers = 1
mesh_max_upper = 1.709176397414741464507875 (attained at j = 1794536 )
Target bound   = 1.709176398000000000000000
Certified mesh_max_upper <= bound : True
\end{verbatim}
The final line \texttt{True} certifies that the computed value \texttt{mesh\_max\_upper} is bounded above by the stated target bound,
and hence establishes Lemma~\ref{lem:mesh}.

\end{document}